\documentclass[
a4paper,
11pt,
twoside,
]{article}
\usepackage[utf8]{inputenc}
\usepackage[T1]{fontenc}
\usepackage[osf,sc,noBBpl]{mathpazo} % don't use mathpazo's bb font (but amsmath's instead): solves PDF/A compliance problem
\usepackage[scaled]{beramono}
\usepackage{microtype}
\usepackage[english]{babel}
\usepackage{geometry}
\usepackage{amsmath,amsfonts,amssymb,amsthm}
\usepackage{etoolbox}
\usepackage{graphicx}
\usepackage[dvipsnames]{xcolor}
\usepackage{mathtools}
\usepackage{siunitx}
\usepackage{booktabs}
\usepackage{enumitem}
\setlist{itemsep=0pt}

\newcommand{\D}{\mathop{}\!\mathrm{d}}

\renewcommand{\restriction}{\raise-.5ex\hbox{\ensuremath{\upharpoonright}}}

\usepackage{fnpct}

\definecolor{webgreen}{rgb}{0,.5,0}
\definecolor{webbrown}{rgb}{.6,0,0}
\definecolor{myblue}{rgb}{0,0.25,0.5}
\usepackage[
  colorlinks=true,
  breaklinks=true,
  bookmarksnumbered,
  pdfhighlight=/O,
  urlcolor=webbrown,
  linkcolor=myblue,
  citecolor=webgreen,
  hyperfootnotes=true,
]{hyperref}
\newcommand{\email}[1]{\href{mailto:#1}{\texttt{#1}}}
\makeatletter
\newcommand{\pagerefstar}{\@pagerefstar}
\makeatother

\usepackage{csquotes}

\usepackage[
  mincrossrefs=99,
  style=numeric-comp,
  bibstyle=siam,
  sorting=nyt,
  backref=true,
  backend=biber,
  url=true,
  doi=true,
  maxbibnames=99,
  maxnames=99,
  giveninits=true,
  useprefix=true,
  date=year,
  autolang=other
]{biblatex}
\DeclareCiteCommand{\citeauthornsc}
  {%
   \boolfalse{citetracker}%
   \boolfalse{pagetracker}%
   \usebibmacro{prenote}}
  {\ifciteindex
     {\indexnames{labelname}}
     {}%
   \printnames{labelname}}
  {\multicitedelim}
  {\usebibmacro{postnote}}
\DeclareSourcemap{
  \maps[datatype=bibtex]{
    \map{
      \step[fieldsource=doi,final]
      \step[fieldset=url,null]
    }
  }
}

\usepackage[capitalize,nameinlink,sort]{cleveref}[0.19]
\crefname{section}{section}{sections}
\crefname{subsection}{subsection}{subsections}
\Crefname{figure}{Figure}{Figures}
\crefformat{equation}{\textup{#2(#1)#3}}
\crefrangeformat{equation}{\textup{#3(#1)#4--#5(#2)#6}}
\crefmultiformat{equation}{\textup{#2(#1)#3}}{ and \textup{#2(#1)#3}}
{, \textup{#2(#1)#3}}{, and \textup{#2(#1)#3}}
\crefrangemultiformat{equation}{\textup{#3(#1)#4--#5(#2)#6}}%
{ and \textup{#3(#1)#4--#5(#2)#6}}{, \textup{#3(#1)#4--#5(#2)#6}}{, and \textup{#3(#1)#4--#5(#2)#6}}
\Crefformat{equation}{#2Equation~\textup{(#1)}#3}
\Crefrangeformat{equation}{Equations~\textup{#3(#1)#4--#5(#2)#6}}
\Crefmultiformat{equation}{Equations~\textup{#2(#1)#3}}{ and \textup{#2(#1)#3}}
{, \textup{#2(#1)#3}}{, and \textup{#2(#1)#3}}
\Crefrangemultiformat{equation}{Equations~\textup{#3(#1)#4--#5(#2)#6}}%
{ and \textup{#3(#1)#4--#5(#2)#6}}{, \textup{#3(#1)#4--#5(#2)#6}}{, and \textup{#3(#1)#4--#5(#2)#6}}
\crefdefaultlabelformat{#2\textup{#1}#3}

\crefname{chapter}{Chapter}{Chapters}
\crefname{appendix}{Appendix}{Appendices}
\crefname{subappendix}{Section}{Sections}
\Crefname{subappendix}{Section}{Sections}
\crefname{page}{page}{pages}
\crefname{section}{Section}{Sections}
\Crefname{section}{Section}{Sections}
\crefname{subsection}{Section}{Sections}
\Crefname{subsection}{Section}{Sections}
\crefalias{subsection}{section}
\crefname{remark}{Remark}{Remarks}
\crefname{footnote}{Footnote}{Footnotes}
\theoremstyle{plain}
\newtheorem{theorem}{Theorem}[section]
\newtheorem{lemma}[theorem]{Lemma}
\newtheorem{proposition}[theorem]{Proposition}
\newtheorem{corollary}[theorem]{Corollary}
\theoremstyle{definition}

\theoremstyle{remark}
\newtheorem{remarkx}[theorem]{Remark}
\newenvironment{remark}
  {\pushQED{\qed}\remarkx}
  {\popQED\endremarkx}

\usepackage[%
norefs,%
nocites%
]{refcheck}
\makeatletter
\newcommand{\refcheckize}[1]{%
  \expandafter\let\csname @@\string#1\endcsname#1%
  \expandafter\DeclareRobustCommand\csname relax\string#1\endcsname[1]{%
    \csname @@\string#1\endcsname{##1}\@for\@temp:=##1\do{\wrtusdrf{\@temp}\wrtusdrf{{\@temp}}}}%
  \expandafter\let\expandafter#1\csname relax\string#1\endcsname
}
\newcommand{\refcheckizetwo}[1]{%
  \expandafter\let\csname @@\string#1\endcsname#1%
  \expandafter\DeclareRobustCommand\csname relax\string#1\endcsname[2]{%
    \csname @@\string#1\endcsname{##1}{##2}\wrtusdrf{##1}\wrtusdrf{{##1}}\wrtusdrf{##2}\wrtusdrf{{##2}}}%
  \expandafter\let\expandafter#1\csname relax\string#1\endcsname
}
\makeatother
\refcheckize{\cref}
\refcheckize{\Cref}
\refcheckizetwo{\crefrange}
\refcheckizetwo{\Crefrange}
\refcheckize{\labelcref}
\refcheckize{\lcnamecref}
\refcheckize{\lcnamecrefs}
\refcheckize{\namecref}
\refcheckize{\namecrefs}
\refcheckize{\cpageref}
\refcheckize{\pagerefstar}
\usepackage{caption}
\usepackage{fancyhdr}
\usepackage{lastpage}

\usepackage[
en-GB,
datesep=/
]{datetime2}
\DTMlangsetup[en-GB]{ord=omit}
\usepackage{amsfonts,amssymb,amsopn}

\renewcommand{\restriction}{\raise-.5ex\hbox{\ensuremath{\upharpoonright}}}

\DeclareMathOperator*{\esssup}{ess\,sup}

\usepackage{mathtools}
\DeclarePairedDelimiter{\abs}{\lvert}{\rvert}
\DeclarePairedDelimiter{\norm}{\lVert}{\rVert}
\DeclarePairedDelimiter{\kernorm}{\lvert\mkern-2mu\lvert\mkern-2mu\lvert}{\rvert\mkern-2mu\rvert\mkern-2mu\rvert}

\usepackage{enumitem}
\newlist{H-hyp}{enumerate}{1}
\setlist[H-hyp]{align=left,leftmargin=*,topsep=1ex,itemsep=0.5ex,label=(H\arabic*),labelindent=1.5ex}
\crefname{H-hypi}{hypothesis}{hypotheses}
\Crefname{H-hypi}{Hypothesis}{Hypotheses}
\newlist{C-hyp}{enumerate}{1}
\setlist[C-hyp]{align=left,leftmargin=*,topsep=1ex,itemsep=0.5ex,label=(C\arabic*),labelindent=1.5ex}
\crefname{C-hypi}{condition}{conditions}
\Crefname{C-hypi}{Condition}{Conditions}
\newlist{statement}{enumerate}{1}
\setlist[statement]{label=\arabic*.,ref=\arabic*}
\crefname{statementi}{part}{parts}
\Crefname{statementi}{Part}{Parts}

\usepackage{bm}

\newcommand{\mytitle}{Lyapunov exponents of renewal equations: numerical approximation and convergence analysis}

\title{\mytitle}

\author{Dimitri Breda, Davide Liessi}

\newcommand{\mydate}{29 March 2024}

\fancypagestyle{mypagestyle}{%
  \fancyhf{}%
  \fancyhead[LO,RE]{\scriptsize \mytitle}%
  \fancyhead[RO,LE]{\scriptsize Breda, Liessi}%
  \fancyfoot[LE,RO]{\scriptsize \thepage\,/\,\pagerefstar{LastPage}}%
  \fancyfoot[LO,RE]{\scriptsize \mydate{}}%
}
\begin{document}

\thispagestyle{empty}

\begin{center}
\LARGE
\mytitle

\bigskip
\large
Dimitri Breda\footnote{\email{dimitri.breda@uniud.it}},
Davide Liessi\footnote{\email{davide.liessi@uniud.it}}

\medskip
\small
CDLab -- Computational Dynamics Laboratory \\
Department of Mathematics, Computer Science and Physics, University of Udine \\
Via delle Scienze 206, 33100 Udine, Italy

\medskip
\large
\mydate
\end{center}

\begin{abstract}
We propose a numerical method for computing the Lyapunov exponents of renewal equations (delay equations of Volterra type), consisting first in applying a discrete QR technique to the associated evolution family suitably posed on a Hilbert state space and second in reducing to finite dimension each evolution operator in the obtained time sequence.
The reduction to finite dimension relies on Fourier projection in the state space and on pseudospectral collocation in the forward time step.
A rigorous proof of convergence of both the discretized operators and the approximated exponents is provided.
A MATLAB implementation is also included for completeness.

\smallskip
\noindent \textbf{Keywords:}
renewal equations;
Lyapunov exponents;
discrete QR method;
evolution family;
Fourier projection;
pseudospectral collocation;
population dynamics.

\smallskip
\noindent \textbf{Mathematical Subject Classification (2020):}
37M25; % Computational methods for ergodic theory (approximation of invariant measures, computation of Lyapunov exponents, entropy, etc.) (in Approximation methods and numerical treatment of dynamical systems)
45D05; % Volterra integral equations
47D99; % None of the above, but in this section (in Groups and semigroups of linear operators, their generalizations and applications)
65R20. % Numerical methods for integral equations
\end{abstract}

\renewcommand{\thefootnote}{\arabic{footnote}}

\section{Introduction}
%\label{sec:intro}

In \cite{BredaLiessi2024} we presented a method to compute the Lyapunov exponents (LEs) or renewal equations (REs), i.e., delay equations of Volterra type, where the rule for extension towards the future based on the known past \cite{DiekmannVerduynLunel2021} prescribes the value of the unknown function itself, instead of the value of the derivative as for delay differential equations (DDEs).
The interest in LEs is motivated by their role in measuring the asymptotic exponential behavior of solutions, which leads to their application in several aspects of the study of dynamics, such as the average asymptotic stability, the insurgence of chaos and the effects of perturbations (see, e.g., \cite{Adrianova1995,DieciVanVleck2002}).

Following the ideas of \cite{BredaDellaSchiava2018,BredaDiekmannGyllenbergScarabelVermiglio2016,ScarabelDiekmannVermiglio2021}, the method of \cite{BredaLiessi2024} was based on the reformulation of the delay equation as an abstract differential equation, discretized via a pseudospectral collocation, yielding an ordinary differential equation (ODE) whose LEs were computed with the standard discrete QR (DQR) method \cite{DieciVanVleck2004,DieciJollyVanVleck2011}.
The resulting method was experimentally effective, exhibiting the expected convergence properties with respect to the final time of the DQR method and providing results compatible with known properties of the test examples.
However, we could not provide a convergence proof for the method, due to the lack of a proof of convergence of the pseudospectral discretization approach for general nonlinear delay equations.

In view of providing a method to compute the LEs of REs with a systematic convergence analysis, this work follows instead the ideas of \cite{BredaVanVleck2014} for DDEs, which adopts a more rigorous approach treating the DDE directly.
There, the equation is posed in a Hilbert space, the associated evolution family is discretized and the DQR method is adapted to the resulting finite-dimensional approximations of the operators.
Moreover, in order to prove the convergence of the method, the DQR method is lifted to infinite dimension and compared with the finite-dimensional approximation.

Since the experimental results obtained with the proposed method are very similar to the ones obtained in \cite{BredaLiessi2024}, in this work we focus on the theoretical foundations of the approach for linear REs%
\footnote{As common in the context of computing LEs, the nonlinear case is treated via linearization along reference trajectories.
As far as REs are concerned, see \cite{BredaLiessi2024} and the comments therein.}
and on the proofs of convergence of the approximations of the evolution operators, on the one hand, and of the LEs, on the other hand.

There are a few fundamental differences with respect to \cite{BredaVanVleck2014}.
First, for the formulation and discretization of the evolution operators we use the pseudospectral approach of \cite{BredaLiessi2018}, which is more versatile in terms of adaptability to other kinds of equations.
Second, the reformulation of REs on an $L^2$ state space (instead of the classical $L^1$ \cite{DiekmannGettoGyllenberg2008}) is less trivial than that of DDEs, since in the latter case the state space is enlarged (from the space of continuous functions), while in the former it is restricted.

As for the convergence of the approximated evolution operators, the proof differs from \cite{BredaLiessi2018} both because of the different state space and of the different goal of the proof: indeed, here we are not interested in the convergence of the spectra of the operators, while the norm convergence of the operators is more important.
The convergence proof of the LEs from \cite{BredaVanVleck2014}, following that of the operators, can instead be mostly preserved.

Finally, in order to prove the (eventual) compactness of the evolution family -- which is necessary to ensure existence of the LEs -- we use the ideas of \cite{BredaLiessi2021}, also based on \cite{GripenbergLondenStaffans1990}, in the new $L^2$ setting.

The overall outcome is a novel numerical method to compute LEs of REs, as effective as the one of \cite{BredaLiessi2024}, but whose convergence is rigorously proved.
To the best of the authors' knowledge, there are no other tools to compute LEs of REs, despite the recognized role of the latter as key ingredients in models of population dynamics (see \cite{DiekmannGettoGyllenberg2008} and the references therein).

The work is structured as follows.
In \cref{sec:re-l2} we formulate the REs on the $L^2$ state space and we prove the existence and uniqueness of the associated initial value problem.
Then, in \cref{sec:evop} we define the family of evolution operators and prove their (eventual) compactness, we define their approximations via pseudospectral discretization and we prove that such approximations are well defined and converge in norm.
In \cref{sec:le} we define the LEs for REs, we describe their approximation both in the infinite- and finite-dimensional setting and we present the proof of their convergence to (part of) the exact LEs.
Finally, in \cref{sec:num} we describe the implementation and provide some results on an example RE, before presenting a few concluding remarks in \cref{sec:concl}.
The MATLAB codes implementing the method and the scripts to reproduce the experiments of \cref{sec:num} are available at \url{http://cdlab.uniud.it/software}.

\section{Renewal equations on \texorpdfstring{$L^2$}{L2}}
\label{sec:re-l2}

For $d \in \mathbb{N}$ and $\tau \in \mathbb{R}$ both positive, consider the function space $X \coloneqq L^{2}([-\tau, 0], \mathbb{R}^{d})$ equipped with the usual $L^{2}$ norm, denoted by $\norm{\cdot}_{X}$.
For $s \in \mathbb{R}$ and a function $x$ defined on $[s - \tau, + \infty)$ let
\begin{equation*}%\label{segment}
x_{t}(\theta) \coloneqq x(t + \theta), \quad t \geq s, \; \theta \in [-\tau, 0].
\end{equation*}
Given a measurable function $C \colon [s , +\infty) \times [-\tau, 0] \rightarrow \mathbb{R}^{d \times d}$, we consider the linear nonautonomous RE%
\footnote{Thanks to the Riesz representation theorem for~$L^{2}$ (see, e.g., \cite[page 400]{RoydenFitzpatrick2010}), every linear nonautonomous RE of type $x(t)=L(t)x_t$ with $L(t)\colon X\to \mathbb{R}^{d}$ linear an bounded can be written in the form~\cref{RE}.}
\begin{equation}\label{RE}
x(t) = \int_{- \tau}^{0} C(t, \theta) x_{t}(\theta) \D \theta, \quad t > s.
\end{equation}
In the following it is convenient to shift the time variable by $s$, obtaining
\begin{equation}\label{RE-s}
x(t) = \int_{- \tau}^{0} C(s+t, \theta) x_{t}(\theta) \D \theta, \quad t > 0.
\end{equation}
There is obviously a one-to-one correspondence between the solutions $x$ of \cref{RE} and $\hat{x}$ of \cref{RE-s}, given by $\hat{x}(t)=x(s+t)$.
Given $\phi \in X$, we define the initial value problems (IVP) for \cref{RE} and \cref{RE-s} by imposing, respectively $x_{s} = \phi$ and $x_{0} = \phi$.

We interpret REs in light of the theory of Volterra equations presented in \cite[Chapter~9]{GripenbergLondenStaffans1990}.
Let $J\coloneqq [0,\tau]$.
The function $K\colon J^2\to\mathbb{R}^{d\times d}$ defined as%
\footnote{\label{fn:drop-s}%
Although $K$ and $g$ in \cref{g-from-phi} depend on $s$, we choose not to reflect this in the notation to keep it lighter. This should not cause any confusion, since in this \lcnamecref{sec:re-l2} $s$ is fixed. We treat the operator $\mathcal{F}$ in \cref{sec:discr,sec:conv-op} in the same way.}
\begin{equation}\label{K}
K(t, \sigma) \coloneqq
\begin{cases}
C(s + t, \sigma - t), & \text{if $\sigma\leq t$}, \\
0, & \text{if $\sigma>t$},
\end{cases}
\end{equation}
is a Volterra kernel on $J$.
If $t\in J$, the IVP defined by \cref{RE-s} and $x_0=\phi$ is equivalent to the Volterra integral equation (VIE) of the second kind
\begin{equation}\label{VIE}
x(t) = \int_{0}^{t} K(t, \sigma) x(\sigma) \D \sigma + g(t)
\end{equation}
with%
\footnote{We have defined $K$ on $J^2$ in order to adhere to the notation of \cite{GripenbergLondenStaffans1990}.
Moreover, in \cref{compactness} we will need to extend the domain of $K$ to $\mathbb{R}^2$ defining $K$ as $0$ outside $J^2$.
As such, we do not use $K$ in the definition of $g$, keeping $C$ instead.}
\begin{equation}\label{g-from-phi}
g(t) \coloneqq \int_{t-\tau}^{0} C(s+t,\sigma-t) \phi(\sigma) \D \sigma.
\end{equation}

Recall that a Volterra kernel $K$ is of type $L^p$ on $J$, with $p\geq1$ or $p=+\infty$, if $\kernorm{K}_{L^{p}(J)}<+\infty$, where%
\footnote{For ease of notation, we omit the codomain of function spaces in the subscript of norms.}
\begin{equation*}
\kernorm{K}_{L^{p}(J)} \coloneqq \sup_{\substack{\norm{f}_{L^q(J)} \leq 1 \\ \norm{g}_{L^p(J)} \leq 1 }} \int_{J} \int_{J} \abs{f(t)K(t, \sigma)g(\sigma)} \D s \D t,
\end{equation*}
$\abs{\cdot}$ is a norm on $\mathbb{R}^{d\times d}$, $q$ satisfies $1/p+1/q=1$ and the supremum is taken over scalar functions $f\in L^q(J,\mathbb{R})$ and $g\in L^p(J,\mathbb{R})$.
A Volterra kernel $R$ is a Volterra resolvent of $K$ on $J$ if
\begin{equation*}
K(t,\sigma) = R(t, \sigma) + \int_J K(t,\rho)R(\rho,\sigma)\D\rho = R(t, \sigma) + \int_J R(t,\rho)K(\rho,\sigma)\D\rho.
\end{equation*}

\begin{theorem}\label{existence-uniqueness-VIE}
If $K\in L^{\infty}(J^2,\mathbb{R}^{d\times d})$ and $g\in L^p(J,\mathbb{R}^{d})$ then \cref{VIE} has a unique solution $x\in L^p(J,\mathbb{R}^{d})$ on $J$, given by the variation of constant formula
\begin{equation}\label{VIE-sol}
x(t) = \int_{0}^{t} R(t, \sigma) g(\sigma) \D \sigma + g(t),
\end{equation}
where $R$ is the resolvent of $K$ on $J$.
\end{theorem}
\begin{proof}
By \cite[Theorem 9.2.7]{GripenbergLondenStaffans1990},
\begin{equation}\label{kernorm}
\kernorm{K}_{L^1(J)} = \esssup_{\sigma \in J} \int_{J} \abs{K(t, \sigma)} \D t,
\quad
\kernorm{K}_{L^{\infty}(J)} = \esssup_{t \in J} \int_{J} \abs{K(t, \sigma)} \D \sigma.
\end{equation}
Both are bounded by $\tau\norm{K}_{L^{\infty}(J^2)}$, which implies that $K$ is a kernel both of type $L^1$ and of type $L^{\infty}$ on $J$.
By \cite[Theorem 9.2.6]{GripenbergLondenStaffans1990} it is also a kernel of type $L^p$ for any $p$.
Moreover, for a subinterval $J_{i} \subset J$ also $\kernorm{K}_{L^1(J_{i})}$ and $\kernorm{K}_{L^{\infty}(J_{i})}$ are both bounded by $\abs{J_{i}}\norm{K}_{L^{\infty}(J^2)}$, where $\abs{J_{i}}$ is the length of the interval $J_{i}$.
By choosing $J_{i}$ small enough, they are thus smaller than $1$.
By \cite[Lemma 9.3.3 and Theorem 9.3.14]{GripenbergLondenStaffans1990}, $K$ has a resolvent $R$ both of type $L^{1}$ and of type $L^{\infty}$ on $J$, and thus, again by \cite[Theorem 9.2.6]{GripenbergLondenStaffans1990}, also of type $L^p$ for any $p$.
We conclude by applying \cite[Theorem 9.3.6]{GripenbergLondenStaffans1990}.
\end{proof}

\begin{corollary}\label{existence-uniqueness-RE}
If $C\in L^{\infty}([s,s+\tau]\times[-\tau,0],\mathbb{R}^{d\times d})$ and $\phi\in X$, then the IVP defined by \cref{RE} and $x_s=\phi$ has a unique solution $x\in L^2([s,s+\tau],\mathbb{R}^{d})$ on $[s,s+\tau]$.
\end{corollary}
\begin{proof}
As already observed above, the solutions of the considered RE and of \cref{RE-s} are in a one-to-one correspondence.
In turn, the IVP defined by \cref{RE-s} and $x_0=\phi$ is equivalent to \cref{VIE} with $K$ and $g$ given, respectively, by \cref{K} and \cref{g-from-phi}.
Since $C$ is in $L^{\infty}$, so is $K$.
Moreover, $\phi$, being in $L^2$, is also in $L^1$, and hence $g$ is in $L^2$:
\begin{equation*}
\begin{aligned}
\norm{g}^2_{L^2(J)} &= \int_{0}^{\tau}\abs[\Big]{\int_{t-\tau}^{0} C(s+t,\sigma-t) \phi(\sigma) \D \sigma}^2 \D t \\
&\leq \int_{0}^{\tau}\Bigl(\int_{t-\tau}^{0} \abs{C(s+t,\sigma-t)} \abs{\phi(\sigma)} \D \sigma \Bigr)^2\D t \\
&\leq \norm{C}^2_{L^{\infty}([s,s+\tau]\times[-\tau,0])} \int_{0}^{\tau}\Bigl(\int_{t-\tau}^{0} \abs{\phi(\sigma)} \D \sigma \Bigr)^2\D t \\
&\leq \norm{C}^2_{L^{\infty}([s,s+\tau]\times[-\tau,0])} \norm{\phi}^2_{L^1([-\tau, 0])} \tau < +\infty.
\end{aligned}
\end{equation*}
By applying \cref{existence-uniqueness-VIE} we obtain the unique solution $\hat{x}\in L^2([0,\tau],\mathbb{R}^{d})$ of \cref{VIE}, and thus of the IVP for \cref{RE-s}, on $[0,\tau]$, which yields the unique solution $x$ of the IVP for \cref{RE} on $[s,s+\tau]$ via $x(t) = \hat{x}(t-s)$.
\end{proof}

\begin{remark}\label{steps}
Assuming that $C\in L^{\infty}_{\text{loc}}([s,+\infty)\times[-\tau,0],\mathbb{R}^{d\times d})$, a reasoning on the lines of Bellman's method of steps~\cite{Bellman1961,BellmanCooke1965} allows to extend the solution of \cref{RE} to any $t > s$, by working successively on $[s+\tau,s+2\tau]$, $[s+2\tau,s+3\tau]$ and so on (see also \cite{BellenZennaro2003,BellmanCooke1963} for similar arguments, and \cite[Section 4.1.2]{Brunner2004} for VIEs).
The solution will be denoted by $x(t)$, or $x(t; s, \phi)$ when emphasis on $s$ and $\phi$ is required.
\end{remark}

Observe that in order to obtain the existence and uniqueness of solutions in $L^2$ (our case of interest), it is enough to require that $C$, and thus $K$, is an $L^2$ function, thanks to \cite[Corollary 9.3.16]{GripenbergLondenStaffans1990}.
We decided to state the previous results requiring instead an $L^{\infty}$ kernel (which allows to draw a stronger conclusion), because it is already needed for the compactness proof that will follow.
Moreover, we recall from%
\footnote{Although the focus of \cite{BredaLiessi2021} is on linear(ized) periodic REs, the result we mention here holds for the general nonautonomous case.}
\cite[Proposition 32]{BredaLiessi2021} that if a nonlinear RE has a $C^1$ and globally Lipschitz continuous right-hand side, the integration kernels of the corresponding linearized REs are $L^{\infty}$ on any compact time interval, so for our purposes the request is not too restrictive.

\section{Evolution operators}
\label{sec:evop}

As illustrated in \cref{sec:defLE}, in order to define the LEs of \cref{RE}, we need to define its family of evolution operators and prove that these are (eventually) compact (\cref{sec:compactness}).
Then, in \cref{sec:discr}, we present the discretization approach that will be used for the evolution operators in the approximation of the LEs in \cref{sec:apprLE}, showing that such discretization is well defined.
Finally, in \cref{sec:conv-op} we prove sufficient conditions for the convergence of the discretized operators (in a sense specified therein), which, again, will be needed for the convergence of the approximated LEs in \cref{sec:conv-exp}.

\subsection{Definition and compactness}
\label{sec:compactness}

Let $\{T(t,s)\}_{t\geq s}$ be the family of linear and bounded evolution operators (see \cite{BredaLiessi2021} for REs and \cite{ChiconeLatushkin1999,DiekmannVanGilsVerduynLunelWalther1995} in general) associated to~\cref{RE}, i.e.,
\begin{equation}\label{T}
T(t, s) \colon X \to X,\quad\quad T(t, s) \phi = x_{t}(\cdot; s, \phi).
\end{equation}
For the proof of the boundedness of $T(t,s)$, see \cref{T-bounded} below.

\begin{theorem}\label{compactness}
Let $s \in \mathbb{R}$, $h\geq\tau$, $\phi \in X$ and $C\in L^{\infty}([s,s+h]\times[-\tau,0],\mathbb{R}^{d\times d})$.
Consider the IVP defined by \cref{RE} and $x_s=\phi$.
The associated evolution operator $T(s+h, s)$ is compact.
\end{theorem}
\begin{proof}
We prove the thesis for \cref{VIE} with \cref{g-from-phi}, which implies the thesis for \cref{RE} with $x_s=\phi$ through \cref{K} and \cref{RE-s}.
Thanks to \cref{existence-uniqueness-VIE,existence-uniqueness-RE,steps}, given \cref{g-from-phi}, \cref{VIE} has a unique solution $x=x(\cdot; \phi)$ on $J\coloneqq[0,\tau]$ given by \cref{VIE-sol}; moreover, from the proof of \cref{existence-uniqueness-VIE}, $K$ is a Volterra kernel of type $L^{1}$ and of type $L^{\infty}$ on $J$ with a resolvent $R$ of type $L^{1}$ and of type $L^{\infty}$ on $J$, which in particular implies that $\kernorm{R}_{L^{\infty}(J^2)}<+\infty$.
Observe also that%
\footnote{The inequalities $\kernorm{R}_{L^{1}(J^2)}<+\infty$ and $\norm{R}_{L^{1}(J^2)} \leq \tau \kernorm{R}_{L^{1}(J^2)}$ hold as well.}
$\norm{R}_{L^{1}(J^2)} \leq \tau \kernorm{R}_{L^{\infty}(J^2)}<+\infty$, which implies that $R\in L^1(J^2, \mathbb{R}^{d\times d})$.
The evolution operator $T(s+h,s)$ of \cref{RE} coincides with the evolution operator $T(h,0)$ of \cref{RE-s} and of \cref{VIE}; let us denote it with $T$ for brevity.

It is convenient to extend the domains of $\phi$ and $g$ to $\mathbb{R}$ and those of $C$, $K$ and $R$ to $\mathbb{R}^2$ by defining them as $0$ outside their original domains; with a little abuse of notation, we use the same symbols to denote the extended functions.
For ease of notation, define also $J_C\coloneqq [s,s+h]\times[-\tau,0]$.
%Define $\phi_{0} \colon \mathbb{R} \to \mathbb{R}^{d}$ as
%\begin{equation}\label{def-phi0}
%\phi_{0}(\theta) \coloneqq
%\begin{cases}
%\phi(\theta), & \text{if } \theta \in [-\tau,0], \\
%0, & \text{otherwise},
%\end{cases}
%\end{equation}
%$g_{0}\colon \mathbb{R} \to \mathbb{R}^{d}$ as
%\begin{equation}\label{def-g0}
%g_{0}(t) \coloneqq
%\begin{cases}
%g(t), & \text{if } t\in [0,\tau], \\
%0, & \text{otherwise},
%\end{cases}
%\end{equation}
%and $K_{0} \colon \mathbb{R}^{2} \rightarrow \mathbb{R}^{d \times d}$ as
%\begin{equation}\label{def-k0}
%K_{0}(t, \sigma) \coloneqq
%\begin{cases}
%K(t, \sigma), & \text{if } t \geq 0 \text{ and } \sigma \in [t-\tau,t], \\
%0, & \text{otherwise}.
%\end{cases}
%\end{equation}

Our aim is to prove that the image of the unit ball in $X$ under $T$ is relatively compact in $X$.
Let $\Phi \coloneqq \{\phi \in X \mid \norm{\phi}_{X} \leq 1\}$ and let $\widehat{T \Phi}$ be the set of prolongations of functions in $T \Phi$ to $\mathbb{R}$ by $0$.
Observe that $\widehat{T \Phi}$ is bounded, since $T$ and $\Phi$ are bounded.
By the Kolmogorov--M.~Riesz--Fréchet theorem (see, e.g., \cite[Theorem 4.26]{Brezis2011}), defining $\tau_{\eta}\psi(\cdot) \coloneqq \psi(\eta + \cdot)$, it is sufficient to prove that
\begin{equation*}
\lim_{\eta \to 0} \norm{\tau_{\eta} \psi - \psi}_{L^{2}(\mathbb{R})} = 0
\end{equation*}
uniformly in $\psi \in \widehat{T \Phi}$, i.e., for each $\epsilon > 0$ there exists $\delta > 0$ such that $\norm{\tau_{\eta} \psi - \psi}_{L^{2}(\mathbb{R})} < \epsilon$ for all $\psi \in \widehat{T \Phi}$ and all $\eta \in \mathbb{R}$ such that $0 < \abs{\eta} < \delta$.
Let $\psi \in \widehat{T \Phi}$.
Then there exists $\phi \in \Phi$ such that for $t \in \mathbb{R}$
\begin{equation*}
\psi(t) =
\begin{cases}
(T \phi)(t), & \text{if } t \in [-\tau, 0], \\
0, & \text{otherwise.}
\end{cases}
\end{equation*}
Observe that $\norm{\psi}_{L^{2}(\mathbb{R})} = \norm{T \phi}_{X}$.
Assume that $\eta > 0$; the case $\eta < 0$ is analogous.
Without loss of generality, suppose $0 < \eta < \tau$.
Then
\begin{equation}\label{split-integral}
\begin{split}
\norm{\tau_{\eta} \psi - \psi}^2_{L^{2}(\mathbb{R})}
&= \int_{-\infty}^{+\infty} \abs{\psi(\theta+\eta) - \psi(\theta)}^2 \D \theta \\
&= \int_{-\tau}^{-\eta} \abs{(T \phi)(\theta+\eta) - (T \phi)(\theta)}^2 \D \theta \\
&\qquad + \int_{-\tau-\eta}^{-\tau} \abs{(T \phi)(\theta+\eta)}^2 \D \theta + \int_{-\eta}^{0} \abs{(T \phi)(\theta)}^2 \D \theta.
\end{split}
\end{equation}

Recalling \cref{VIE-sol,T}, for $\theta \in [-\tau, 0]$
\begin{equation*}%\label{uphi-with-resolvent}
(T \phi)(\theta) = \int_{0}^{h+\theta} R(t+\theta, \sigma) g(\sigma) \D \sigma + g(t+\theta),
\end{equation*}
hence the three terms in the right-hand side of \cref{split-integral} are equal, respectively, to
\begin{align}
\begin{split}
&\int_{-\tau}^{-\eta} \abs[\Big]{\int_{0}^{h+\theta+\eta} R(h+\theta+\eta, \sigma) g(\sigma) \D \sigma + g(h+\theta+\eta) \\
&\qquad\qquad\qquad - \int_{0}^{h+\theta} R(h+\theta, \sigma) g(\sigma) \D \sigma - g(h+\theta)}^2 \D \theta,
\end{split} \label{A} \\[2ex]
&\int_{-\tau}^{-\tau+\eta} \abs[\Big]{\int_{0}^{h+\theta} R(h+\theta, \sigma) g(\sigma) \D \sigma + g(h+\theta)}^2 \D \theta \label{B} \\
\intertext{and}
&\int_{-\eta}^{0} \abs[\Big]{\int_{0}^{h+\theta} R(h+\theta, \sigma) g(\sigma) \D \sigma + g(h+\theta)}^2 \D \theta. \label{C}
\end{align}

By using the triangle inequality and expanding the square, \cref{C} is less than or equal to
\begin{equation*}
\begin{split}
&\int_{-\eta}^{0} \Bigl(\abs[\Big]{\int_{0}^{h+\theta} R(h+\theta, \sigma) g(\sigma) \D \sigma}^2
+\abs{g(h+\theta)}^2 \\
&\qquad\qquad\qquad+2\abs[\Big]{\int_{0}^{h+\theta} R(h+\theta, \sigma) g(\sigma) \D \sigma}\abs{g(h+\theta)}\Bigr) \D \theta.
\end{split}
\end{equation*}
Using the definition \cref{g-from-phi} of $g$ and recalling \cite[Theorem 9.2.7]{GripenbergLondenStaffans1990} (see \cref{kernorm}), we obtain
\begin{align*}
&\int_{-\eta}^{0} \abs[\Big]{\int_{0}^{h+\theta} R(h+\theta, \sigma) g(\sigma) \D \sigma}^2 \D \theta \\
&\leq \int_{-\eta}^{0} \abs[\Big]{\int_{0}^{h+\theta} R(h+\theta, \sigma) \int_{\sigma-\tau}^{0} C(s+\sigma,\rho-\sigma) \phi(\rho) \D \rho \D \sigma}^2 \D \theta \\
&\leq \int_{-\eta}^{0} \Bigl(\int_{0}^{h+\theta} \abs{R(h+\theta, \sigma)} \int_{\sigma-\tau}^{0} \abs{C(s+\sigma,\rho-\sigma)} \abs{\phi(\rho)} \D \rho \D \sigma\Bigr)^2 \D \theta \\
&\leq \norm{C}^2_{L^{\infty}(J_C)}\int_{-\eta}^{0} \Bigl(\int_{0}^{h+\theta} \abs{R(h+\theta, \sigma)} \int_{\sigma-\tau}^{0} \abs{\phi(\rho)} \D \rho \D \sigma\Bigr)^2 \D \theta \\
&\leq \norm{C}^2_{L^{\infty}(J_C)}\norm{\phi}^2_{X} \int_{-\eta}^{0} \Bigl(\int_{0}^{h+\theta} \abs{R(h+\theta, \sigma)} \D \sigma\Bigr)^2 \D \theta \\
&\leq \norm{C}^2_{L^{\infty}(J_C)}\norm{\phi}^2_{X} \kernorm{R}^2_{L^{\infty}(J^2)} \eta.
\end{align*}
As for the second term, we obtain
\begin{align*}
\int_{-\eta}^{0} \abs{g(h+\theta)}^2 \D \theta
&\leq\int_{-\eta}^{0} \abs[\Big]{\int_{h+\theta-\tau}^{0} C(s+h+\theta,\sigma-h-\theta) \phi(\sigma) \D \sigma}^2 \D \theta \\
&\leq\int_{-\eta}^{0} \Bigl(\int_{h+\theta-\tau}^{0} \abs{C(s+h+\theta,\sigma-h-\theta)} \abs{\phi(\sigma)} \D \sigma\Bigr)^2 \D \theta \\
&\leq\norm{C}^2_{L^{\infty}(J_C)}\norm{\phi}^2_{X} \eta.
\end{align*}
Finally, for the third term we obtain, with similar computations,
\begin{align*}
&\int_{-\eta}^{0} 2\abs[\Big]{\int_{0}^{h+\theta} R(h+\theta, \sigma) g(\sigma) \D \sigma}\abs{g(h+\theta)}\D \theta \\
&\leq2\int_{-\eta}^{0} \abs[\Big]{\int_{0}^{h+\theta} R(h+\theta, \sigma) \int_{\sigma-\tau}^{0} C(s+\sigma,\rho-\sigma) \phi(\rho) \D \rho \D \sigma} \\
&\qquad\qquad\qquad\cdot\abs[\Big]{\int_{h+\theta-\tau}^{0} C(s+h+\theta,\sigma-h-\theta) \phi(\sigma) \D \sigma}\D \theta \\
&\leq \norm{C}^2_{L^{\infty}(J_C)}\norm{\phi}^2_{X} \kernorm{R}_{L^{\infty}(J^2)} \eta.
\end{align*}
All these terms converge to $0$ uniformly in $\phi\in\Phi$ as $\eta\to0$.
The same computations prove the uniform convergence to $0$ of \cref{B}.

By splitting the first inner integral, using the triangle inequality and expanding the square, we obtain that \cref{A} is less than or equal to
\begin{equation}\label{A-bound}
\begin{aligned}
&\int_{-\tau}^{-\eta} \Bigl( \abs[\Big]{\int_{h+\theta}^{h+\theta+\eta} R(h+\theta+\eta, \sigma) g(\sigma) \D \sigma}^2 \\
&\qquad + \abs[\Big]{\int_{0}^{h+\theta} (R(h+\theta+\eta, \sigma) - R(h+\theta, \sigma)) g(\sigma) \D \sigma}^2 \\
&\qquad + \abs{g(h+\theta+\eta) - g(h+\theta)}^2 \\
&\qquad + 2\abs[\Big]{\int_{h+\theta}^{h+\theta+\eta} R(h+\theta+\eta, \sigma) g(\sigma) \D \sigma}\abs[\Big]{\int_{0}^{h+\theta} (R(h+\theta+\eta, \sigma) - R(h+\theta, \sigma)) g(\sigma) \D \sigma} \\
&\qquad + 2\abs[\Big]{\int_{h+\theta}^{h+\theta+\eta} R(h+\theta+\eta, \sigma) g(\sigma) \D \sigma}\abs{g(h+\theta+\eta) - g(h+\theta)}^2  \\
&\qquad + 2\abs[\Big]{\int_{0}^{h+\theta} (R(h+\theta+\eta, \sigma) - R(h+\theta, \sigma)) g(\sigma) \D \sigma}\abs{g(h+\theta+\eta) - g(h+\theta)}^2 \Bigr) \D \theta.
\end{aligned}
\end{equation}
For the first term, we obtain
\begin{equation}\label{A-bound-1}
\begin{aligned}
&\int_{-\tau}^{-\eta} \abs[\Big]{\int_{h+\theta}^{h+\theta+\eta} R(h+\theta+\eta, \sigma) g(\sigma) \D \sigma}^2 \D \theta \\
&\leq \int_{-\tau}^{-\eta} \Bigl( \int_{h+\theta}^{h+\theta+\eta} \abs{R(h+\theta+\eta, \sigma)} \int_{\sigma-\tau}^{0} \abs{C(s+\sigma,\rho-\sigma)} \abs{\phi(\rho)} \D \rho \D \sigma\Bigr)^2 \D \theta \\
&\leq \norm{C}^2_{L^{\infty}(J^2)}\norm{\phi}^2_{X} \int_{-\tau}^{-\eta} \Bigl( \int_{h+\theta}^{h+\theta+\eta} \abs{R(h+\theta+\eta, \sigma)} \D \sigma\Bigr)^2 \D \theta \\
&\leq \norm{C}^2_{L^{\infty}(J^2)}\norm{\phi}^2_{X} \int_{-\tau}^{-\eta} \Bigl( \int_{h+\theta}^{h+\theta+\eta} \abs{R(h+\theta, \sigma)} \D \sigma \\
&\quad+ \int_{h+\theta}^{h+\theta+\eta} \abs{R(h+\theta+\eta, \sigma)-R(h+\theta, \sigma)} \D \sigma\Bigr)^2 \D \theta \\
&\leq \norm{C}^2_{L^{\infty}(J^2)}\norm{\phi}^2_{X} \Bigl[ \int_{-\tau}^{-\eta} \Bigl( \int_{h+\theta}^{h+\theta+\eta} \abs{R(h+\theta, \sigma)} \D \sigma\Bigr)^2 \D \theta \\
&\quad+ \int_{-\tau}^{-\eta} \Bigl( \int_{h+\theta}^{h+\theta+\eta} \abs{R(h+\theta+\eta, \sigma)-R(h+\theta, \sigma)} \D \sigma\Bigr)^2 \D \theta \\
&\quad+2\int_{-\tau}^{-\eta} \Bigl( \int_{h+\theta}^{h+\theta+\eta} \abs{R(h+\theta, \sigma)} \D \sigma\Bigr)\Bigl( \int_{h+\theta}^{h+\theta+\eta} \abs{R(h+\theta+\eta, \sigma)-R(h+\theta, \sigma)} \D \sigma\Bigr) \D \theta\Bigr].
\end{aligned}
\end{equation}
Considering the first addend and using H\"older's inequality, we obtain
\begin{align*}
&\int_{-\tau}^{-\eta} \Bigl( \int_{h+\theta}^{h+\theta+\eta} \abs{R(h+\theta, \sigma)} \D \sigma\Bigr)^2 \D \theta \\
&\leq \int_{\mathbb{R}} \Bigl( \int_{\mathbb{R}} \mathbf{1}_{[h+\theta,h+\theta+\eta]}(\sigma)\abs{R(h+\theta, \sigma)} \D \sigma\Bigr)^2 \D \theta \\
&\leq \norm[\Big]{\int_{\mathbb{R}} \mathbf{1}_{[h+\theta,h+\theta+\eta]}(\sigma)\abs{R(h+\theta, \sigma)} \D \sigma}_{L^{\infty}(\mathbb{R})} \norm[\Big]{\int_{\mathbb{R}} \mathbf{1}_{[h+\theta,h+\theta+\eta]}(\sigma)\abs{R(h+\theta, \sigma)} \D \sigma}_{L^{1}(\mathbb{R})}\\
&\leq \kernorm{R}_{L^{\infty}(J^2)} \int_{\mathbb{R}} \int_{h+\theta}^{h+\theta+\eta}\abs{R(h+\theta, \sigma)} \D \sigma\D\theta,
\end{align*}
which converges to $0$ as $\eta\to0$ because $R$ is an $L^{1}$ function.
Similarly, for the second addend in \cref{A-bound-1} we obtain
\begin{align*}
&\int_{-\tau}^{-\eta} \Bigl( \int_{h+\theta}^{h+\theta+\eta} \abs{R(h+\theta+\eta, \sigma)-R(h+\theta, \sigma)} \D \sigma\Bigr)^2 \D \theta \\
&\leq \kernorm{R(\eta+\cdot,\cdot)-R}_{L^{\infty}(J^2)} \int_{\mathbb{R}} \int_{h+\theta}^{h+\theta+\eta} \abs{R(h+\theta+\eta, \sigma)-R(h+\theta, \sigma)} \D \sigma\D\theta \\
&\leq \kernorm{R(\eta+\cdot,\cdot)-R}_{L^{\infty}(J^2)} \int_{\mathbb{R}^2} \abs{R(t+\eta, \sigma)-R(t, \sigma)} \D \mu,
\end{align*}
where $\mu$ is the Lebesgue measure on $\mathbb{R}^{2}$ and we used Fubini's theorem.
%https://terrytao.wordpress.com/2010/10/16/245a-notes-5-differentiation-theorems/#trans-cts
This converges to $0$ thanks to the continuity of translation in $L^{1}(\mathbb{R}^2,\mathbb{R}^{d\times d})$ and because $R(\eta+\cdot,\cdot)$ is a Volterra kernel of type $L^{\infty}$ since $\eta$ is positive%
\footnote{If $\eta<0$ we can shift the integration variables to obtain $R-R(-\eta+\cdot,\cdot)$.}.
The third addend in \cref{A-bound-1} is treated in the same way, and it is indifferent which factor is chosen for the $L^1$ and $L^{\infty}$ norms.
All these limits hold independently of $\phi\in\Phi$.
The second term in \cref{A-bound} is treated similarly.
For the third term, let us first observe, similarly to the proof of \cref{existence-uniqueness-RE}, that $g$ is in $L^{\infty}$:
\begin{equation*}
\begin{aligned}
\norm{g}_{L^{\infty}(J)} &= \esssup_{t\in J}\abs[\Big]{\int_{t-\tau}^{0} C(s+t,\sigma-t) \phi(\sigma) \D \sigma} \\
&\leq \esssup_{t\in J}\int_{-\tau}^{0} \abs{C(s+t,\sigma-t)} \abs{\phi(\sigma)} \D \sigma \\
&\leq \norm{C}_{L^{\infty}(J_C)} \norm{\phi}_{L^1([-\tau, 0])} < +\infty,
\end{aligned}
\end{equation*}
where we used the fact that $\phi$ is in $L^1$, being in $L^2$.
We obtain
\begin{align*}
&\int_{-\tau}^{-\eta} \abs{g(h+\theta+\eta) - g(h+\theta)}^2 \D \theta \\
&\leq\int_{\mathbb{R}} \abs{g(h+\theta+\eta) - g(h+\theta)}^2 \D \theta \\
&\leq\norm{g(\eta+\cdot) - g}_{L^{\infty}(\mathbb{R})}\norm{g(\eta+\cdot) - g}_{L^{1}(\mathbb{R})},
\end{align*}
%https://terrytao.wordpress.com/2010/10/16/245a-notes-5-differentiation-theorems/#trans-cts
which converges to $0$ as $\eta\to0$ independently of $\phi\in\Phi$ thanks to the continuity of translation in $L^{1}(\mathbb{R},\mathbb{R})$.
The remaining terms in \cref{A-bound} are treated similarly.
\end{proof}

\subsection{Discretization}
\label{sec:discr}

Let $h \in \mathbb{R}$ be positive and define
\begin{equation}\label{Tsh}
T \coloneqq T(s + h, s).
\end{equation}
Inspired by the approach of \cite{BredaMasetVermiglio2012,BredaLiessi2018} (more recent and versatile than that of \cite{BredaVanVleck2014}), we reformulate $T$ as follows.

Define the function spaces $X^{+} \coloneqq L^{2}([0, h], \mathbb{R}^{d})$ and $X^{\pm} \coloneqq L^{2}([-\tau, h], \mathbb{R}^{d})$, equipped with the corresponding $L^{2}$ norms denoted, respectively, by $\norm{\cdot}_{X^{+}}$ and $\norm{\cdot}_{X^{\pm}}$.
Define the operator $V \colon X \times X^{+} \to X^{\pm}$ as
\begin{equation*}%\label{def_V}
V(\phi, w)(t) \coloneqq
\begin{cases}
w(t), & \quad t \in (0, h], \\
\phi(t), & \quad t \in [-\tau, 0].
\end{cases}
\end{equation*}
Let also $V^{-} \colon X \to X^{\pm}$ and $V^{+} \colon X^{+} \to X^{\pm}$ be given, respectively, by $V^{-} \phi \coloneqq V(\phi, 0_{X^{+}})$ and $V^{+} w \coloneqq V(0_{X}, w)$, where $0_{Y}$ denotes the null element of a linear space $Y$ (similarly, $I_{Y}$ in the sequel stands for the identity operator in $Y$).
Observe that
\begin{equation}\label{decomp_V}
V(\phi, w) = V^{-} \phi + V^{+} w.
\end{equation}
Define also the operator $\mathcal{F} \colon X^{\pm} \to X^{+}$ as%
\footnote{See \cref{fn:drop-s}.}
\begin{equation*}%\label{Fs}
\mathcal{F} v(t) \coloneqq \int_{- \tau}^{0} C(s+t, \theta) v(t+\theta) \D \theta, \quad\quad t \in [0, h].
\end{equation*}
Finally, the evolution operator $T$ is reformulated as
\begin{equation}\label{T-as-V}
T \phi = V(\phi, w^{\ast})_{h},
\end{equation}
where $w^{\ast} \in X^{+}$ is the solution of the fixed point equation 
\begin{equation}\label{fixed_point}
w = \mathcal{F} V(\phi, w),
\end{equation}
which exists uniquely thanks to~\cref{existence-uniqueness-RE}.

\bigskip

The following discretization of the infinite-dimensional operator $T$ is inspired by \cite{BredaVanVleck2014}; the change from the one in \cite{BredaLiessi2018} is motivated by the choice of a state space of $L^2$ functions.
Let $\{\psi_i\}_{i\in\mathbb{N}}$ and $\{\psi^+_i\}_{i\in\mathbb{N}}$ be, respectively, the families of Legendre polynomials on $[-\tau,0]$ and $[0,h]$.
Let $M$ and $N$ be positive integers and let $0\leq \theta^+_1<\dots\theta^+_N\leq h$ be the zeros of $\psi^+_N$.
The discrete counterparts of $X$ and $X^+$ are, respectively, $X_{M} \coloneqq (\mathbb{R}^{d})^{M+1}$ and $X_{N}^{+} \coloneqq (\mathbb{R}^{d})^{N}$.
The restriction operator $R_{M}\colon X \to X_{M}$ is defined as $R_{M}\phi\coloneqq (\phi_0,\dots,\phi_M)$, with $\phi_i$ being the coefficients of the Fourier projection of $\phi$ on the Legendre polynomials, i.e., $\phi = \sum_{i=0}^{+\infty}\phi_i\psi_i$.
The prolongation operator $P_{M}\colon X_{M} \to X$ is defined as $P_{M}\Phi\coloneqq \sum_{i=0}^{M}\Phi_i\psi_i$, with $\Phi=(\Phi_0,\dots,\Phi_M)\in X_M$.
Consider a subspace $\widetilde{X}^+$ of $X^+$ regular enough to make point-wise evaluation meaningful.
The restriction operator $R_{N}^{+} \colon \widetilde{X}^{+} \to X_{N}^{+}$ is given by $R_{N}^{+} w \coloneqq (w(\theta^+_{1}), \dots, w(\theta^+_{N}))$.
The prolongation operator $P_{N}^{+} \colon X_{N}^{+} \to X^{+}$ is the discrete Lagrange interpolation operator
$P_{N}^{+} W(t) \coloneqq \sum_{i = 1}^{N} \ell_{i}^{+}(t) W_{i}$,
where $t \in [0, h]$, $W=(W_1,\dots,W_N)$, and $\ell_{1}^{+}, \dots, \ell_{N}^{+}$ are the Lagrange polynomials on the nodes $\{\theta^+_i\}_{i\in\{1,\dots,N\}}$.
Observe that
\begin{equation*}%\label{PR-RP}
R_{M} P_{M} = I_{X_{M}}, \quad\quad
P_{M} R_{M} = F_{M},
\end{equation*}
where $F_{M} \colon X \to X$ is the Fourier projection operator on $\{\psi_i\}_{i\in\{0,\dots,M\}}$, while
\begin{equation}\label{RP-PR-p}
R_{N}^{+} P_{N}^{+} = I_{X_{N}^{+}}, \quad\quad
P_{N}^{+} R_{N}^{+} = \mathcal{L}_{N}^{+},
\end{equation}
where $\mathcal{L}_{N}^{+} \colon \widetilde{X}^{+} \to X^{+}$ is the Lagrange interpolation operator on $\{\theta^+_i\}_{i\in\{1,\dots,N\}}$.
Observe that the range of $P_{N}^{+}$ and $\mathcal{L}_{N}^{+}$ is the space of polynomials of degree at most $N-1$.

Following~\cref{T-as-V,fixed_point}, the discretization of indices $M$ and $N$ of the evolution operator $T$ in~\cref{T} is the finite-dimensional operator $\mathbf{T}_{M, N} \colon X_{M} \to X_{M}$ defined as
\begin{equation}\label{TMNbold}
\mathbf{T}_{M, N} \Phi \coloneqq R_{M} V(P_{M} \Phi, P_{N}^{+} W^{\ast})_{h},
\end{equation}
where $W^{\ast} \in X_{N}^{+}$ is a solution of the fixed point equation
\begin{equation}\label{discrete_FP}
W = R_{N}^{+} \mathcal{F} V(P_{M} \Phi, P_{N}^{+} W)
\end{equation}
for the given $\Phi \in X_{M}$.

\bigskip

We now establish that \cref{discrete_FP} is well-posed following \cite{BredaLiessi2018}.
More precisely, given $\phi \in X$, we consider the collocation equation
\begin{equation}\label{collocation}
W = R_{N}^{+} \mathcal{F} V(\phi, P_{N}^{+} W)
\end{equation}
in $W \in X_{N}^{+}$, we show that it has a unique solution and we study its relation to the unique solution $w^{\ast} \in X^{+}$ of~\cref{fixed_point}.

Using~\cref{decomp_V}, the equations~\cref{fixed_point,collocation} can be rewritten, respectively, as
\begin{equation}\label{fixed-point-2}
(I_{X^{+}} - \mathcal{F} V^{+}) w = \mathcal{F} V^{-} \phi
\end{equation}
and
\begin{equation}\label{collocation_2}
(I_{X_{N}^{+}} - R_{N}^{+} \mathcal{F} V^{+} P_{N}^{+}) W = R_{N}^{+} \mathcal{F} V^{-} \phi.
\end{equation}
The following preliminary result concerns the operators
\begin{equation}\label{cont_coll_op}
I_{X^{+}} - \mathcal{L}_{N}^{+} \mathcal{F} V^{+} \colon X^{+} \to X^{+},
\end{equation}
and
\begin{equation}\label{discr_coll_op}
I_{X_{N}^{+}} - R_{N}^{+} \mathcal{F} V^{+} P_{N}^{+} \colon X_{N}^{+} \to X_{N}^{+}.
\end{equation}

\begin{proposition}[\protect{\cite[Proposition 4.2]{BredaLiessi2018}}]\label{discr-cont-coll_eq}
If the operator~\cref{cont_coll_op} is invertible, then the operator~\cref{discr_coll_op} is invertible.
Moreover, given $\bar{W} \in X_{N}^{+}$, the unique solution $\hat{w} \in X^{+}$ of
\begin{equation}\label{cont_coll_eq}
(I_{X^{+}} - \mathcal{L}_{N}^{+} \mathcal{F} V^{+}) w  = P_{N}^{+} \bar{W}
\end{equation}
and the unique solution $\hat{W} \in X_{N}^{+}$ of
\begin{equation}\label{discr_coll_eq}
(I_{X_{N}^{+}} - R_{N}^{+} \mathcal{F} V^{+} P_{N}^{+}) W = \bar{W}
\end{equation}
are related by $\hat{W} = R_{N}^{+} \hat{w}$ and $\hat{w} = P_{N}^{+} \hat{W}$.
\end{proposition}

As observed above, \cref{collocation} is equivalent to~\cref{collocation_2}, hence, by choosing
\begin{equation}\label{Wbar}
\bar{W} = R_{N}^{+} \mathcal{F} V^{-} \phi,
\end{equation}
it is equivalent to~\cref{discr_coll_eq}.
Observe also that thanks to~\cref{RP-PR-p} the equation
\begin{equation}\label{collocation_hat}
w = \mathcal{L}_{N}^{+} \mathcal{F} V(\phi, w)
\end{equation}
can be rewritten as $(I_{X^{+}} - \mathcal{L}_{N}^{+} \mathcal{F} V^{+}) w
= \mathcal{L}_{N}^{+} \mathcal{F} V^{-} \phi
= P_{N}^{+} R_{N}^{+} \mathcal{F} V^{-} \phi$,
%\begin{equation*}
%(I_{X^{+}} - \mathcal{L}_{N}^{+} \mathcal{F} V^{+}) w
%= \mathcal{L}_{N}^{+} \mathcal{F} V^{-} \phi
%= P_{N}^{+} R_{N}^{+} \mathcal{F} V^{-} \phi,
%\end{equation*}
which is equivalent to~\cref{cont_coll_eq} with the choice~\cref{Wbar}.
Thus, by \cref{discr-cont-coll_eq}, if the operator~\cref{cont_coll_op} is invertible, then the equation~\cref{collocation} has a unique solution $W^{\ast} \in X_{N}^{+}$ such that
\begin{equation}\label{stars-vs-hats}
W^{\ast} = R_{N}^{+} w^{\ast}_{N}, \qquad
w^{\ast}_{N} = P_{N}^{+} W^{\ast},
\end{equation}
where $w^{\ast}_{N} \in X^{+}$ is the unique solution of~\cref{collocation_hat}.
Note for clarity that~\cref{Wbar} implies $w^{\ast}_{N} = \hat{w}$ with~$\hat{w}$ as in \cref{discr-cont-coll_eq}.
So, now we show that~\cref{cont_coll_op} is invertible under due assumptions.

Consider the Sobolev space%
\footnote{We define the Sobolev space on the closed interval as the one on the open interval with its element prolonged to the interval endpoints by continuity, see \cite[Theorem 8.2]{Brezis2011}.}
$H^{1,+} \coloneqq H^1([0,h],\mathbb{R}^{d}) = W^{1,2}([0,h],\mathbb{R}^{d})$ equipped with the norm defined by $\norm{w}^2_{H^{1,+}} \coloneqq \norm{w}^2_{X^+}+\norm{w'}^2_{X^+}$, where $w'$ is the weak derivative of $w$.
Recall that $H^{1,+}\subset L^2([0,h],\mathbb{R}^{d}) = X^+$ and $H^{1,+}\subset C([0,h],\mathbb{R}^{d})$ (see, e.g., \cite[Theorem 8.8]{Brezis2011}).
In particular, $H^{1,+}$ is a suitable choice for $\widetilde{X}^+$.
In the following, some hypotheses on $\mathcal{F}$ and $V$ are needed beyond the one on $C$ of \cref{existence-uniqueness-RE}, in order to attain the regularity required to ensure the convergence of the method, namely
\begin{H-hyp}
\item\label{hyp-C} $C\in L^{\infty}([s,s+h]\times[-\tau,0],\mathbb{R}^{d\times d})$ (cf.\ \cref{existence-uniqueness-RE});
\item\label{FsVp-to-AC} $\mathcal{F} V^{+} \colon X^{+} \to X^{+}$ has range contained in $H^{1,+}$ and $\mathcal{F} V^{+} \colon X^{+} \to H^{1,+}$ is bounded;
\item\label{FXsVm-to-AC} $\mathcal{F} V^{-} \colon X \to X^{+}$ has range contained in $H^{1,+}$ and $\mathcal{F} V^{-} \colon X \to H^{1,+}$ is bounded.
\end{H-hyp}

Note that \cref{hyp-C,FsVp-to-AC} are used here to prove that the discretization $\mathbf{T}_{M,N}$ is well defined, while \cref{FXsVm-to-AC} will be used in \cref{sec:conv-op} where the convergence to $T$ is discussed.

\begin{lemma}\label{bounded-inverse-fp-eq}
Assuming \cref{hyp-C}, the operator $I_{X^{+}} - \mathcal{F} V^{+}$ is invertible with bounded inverse and~\cref{fixed_point} admits a unique solution in~$X^{+}$.
\end{lemma}
\begin{proof}
The invertibility of $I_{X^{+}} - \mathcal{F} V^{+}$ follows from \cref{existence-uniqueness-RE}.
Its boundedness is clear from the boundedness of $I_{X^{+}}$, $V^{+}$ and $\mathcal{F}$, the latest being obtained via the Cauchy--Schwarz inequality and observing that $C$ is also $L^2$.
The bounded inverse theorem completes the proof.
\end{proof}

\begin{corollary}\label{T-bounded}
Assuming \cref{hyp-C}, $T$ in \cref{Tsh} is bounded.
\end{corollary}
\begin{proof}
Thanks to \cref{T-as-V,decomp_V} we have $T\phi=(V^-\phi)_h+(V^+w^*)_h$ and, recalling \cref{fixed-point-2} and using \cref{bounded-inverse-fp-eq}, $w^*=(I_{X^+}-\mathcal{F}V^+)^{-1}\mathcal{F}V^-\phi$.
Observe that $\norm{(V^-\phi)_h}_{X}\leq \norm{\phi}_{X}$ and $\norm{(V^+w^*)_h}_{X}\leq \norm{w^*}_{X^+}$.
The thesis follows from \cref{bounded-inverse-fp-eq} and the boundedness of $V^-$ and $\mathcal{F}$ (also shown in the proof of \cref{bounded-inverse-fp-eq}).
\end{proof}

\begin{proposition}\label{cont_coll_op-invertible}
Assuming \cref{hyp-C,FsVp-to-AC}, there exists a positive integer $N_{0}$ such that, for any $N \geq N_{0}$, the operator~\cref{cont_coll_op} is invertible and
\begin{equation*}
\norm{(I_{X^{+}} - \mathcal{L}_{N}^{+} \mathcal{F} V^{+})^{- 1}}_{X^{+} \leftarrow X^{+}} \leq 2 \norm{(I_{X^{+}} - \mathcal{F} V^{+})^{- 1}}_{X^{+} \leftarrow X^{+}}.
\end{equation*}
Moreover, for each $\phi \in X$, \cref{collocation_hat} has a unique solution $w^{\ast}_{N} \in X^{+}$ and
\begin{equation*}
\norm{w^{\ast}_{N} - w^{\ast}}_{X^{+}} \leq 2 \norm{(I_{X^{+}} - \mathcal{F} V^{+})^{- 1}}_{X^{+} \leftarrow X^{+}} \norm{\mathcal{L}_{N}^{+} w^{\ast} - w^{\ast}}_{X^{+}},
\end{equation*}
where $w^{\ast} \in X^{+}$ is the unique solution of~\cref{fixed_point}.
\end{proposition}
\begin{proof}
By~\cite[(9.4.24)]{CanutoHussainiQuarteroniZang1988}, if $w \in H^{1,+}$, there exists a constant $c$ such that
$\norm{(\mathcal{L}_{N}^{+} - I) w}_{X^{+}} \leq c (N-1)^{-\frac{1}{2}}\norm{w}_{H^{1,+}}$, which implies that
\begin{equation}\label{LN-I}
\norm{(\mathcal{L}_{N}^{+} - I) \restriction_{H^{1,+}}}_{X^{+} \leftarrow H^{1,+}} \leq c (N-1)^{-\frac{1}{2}}\xrightarrow[N \to +\infty]{} 0.
\end{equation}%
The proof continues as in \cite[Proposition 4.3]{BredaLiessi2018}, with the only difference that here we use \cref{bounded-inverse-fp-eq} for the existence and boundedness of the inverse of $I_{X^{+}} - \mathcal{F} V^{+}$.
\end{proof}

Under \cref{hyp-C,FsVp-to-AC} we conclude that \cref{discrete_FP} is well posed and hence the discretization of $T$ defined in \cref{TMNbold} is well defined.
As a last remark, we translate \cref{FsVp-to-AC,FXsVm-to-AC} in terms of requirements on the kernel $C$ defining the RE.
Indeed, using the characterization of $H^{1,+}$ given by \cite[Proposition 8.3]{Brezis2011}, it is possible to prove that \cref{FsVp-to-AC,FXsVm-to-AC} are fulfilled if the following condition on the kernel $C$ of \cref{RE} holds:
\begin{quote}
the function $t\mapsto C(s+t,\sigma-t)$ satisfies the condition (ii) of \cite[Proposition 8.3]{Brezis2011} uniformly with respect to $\sigma\in[-\tau,h]$, i.e., there exists a constant $k$ such that for each $\phi \in  C^1_c((0,h))$ and almost all $\sigma \in[-\tau,h]$
\begin{equation*}
\abs[\Big]{\int_0^h C(s+t,\sigma-t)\phi'(t)\D t}\leq k\norm{\phi}_{X^+}.
\end{equation*}
\end{quote}
Thanks again to the characterization given by \cite[Proposition 8.3]{Brezis2011}, for a fixed $\sigma$ this condition is equivalent to $t\mapsto C(s+t,\sigma-t)$ being in $H^{1,+}$.

\subsection{Convergence}
\label{sec:conv-op}

In view of discussing the error in approximating the LEs in \cref{sec:conv-exp} we provide here a convergence analysis of the discretization of $T$ introduced in \cref{sec:discr}.
The operators $T$ and $\mathbf{T}_{M, N}$ are defined on different spaces and cannot be compared directly.
We thus introduce the finite-rank operator $T_{M, N} \colon X \to X$ associated to $\mathbf{T}_{M, N}$, defined as $T_{M, N} \coloneqq P_{M} \mathbf{T}_{M, N} R_{M}$.

Let $h\geq \tau$.
Define the operator $T_{N} \colon X \to X$ as
\begin{equation*}
T_{N} \phi \coloneqq V(\phi, w^{\ast}_{N})_{h},
\end{equation*}
where $w^{\ast}_{N} \in X^{+}$ is the solution of the fixed point equation~\cref{collocation_hat}, which, under \cref{hyp-C,FsVp-to-AC}, exists and is unique thanks to \cref{discr-cont-coll_eq,cont_coll_op-invertible}.
Observe that $w^{\ast}_{N}$ is a polynomial, hence, in particular, $w^{\ast}_{N} \in H^{1,+}$.
Then, for $\phi \in X$, by~\cref{stars-vs-hats},
\begin{equation}\label{TMN-with-FM}
\begin{split}
T_{M, N} \phi
&= P_{M} \mathbf{T}_{M, N} R_{M} \phi \\
&= P_{M} R_{M} V(P_{M} R_{M} \phi, P_{N}^{+} W^{\ast})_{h} \\
&= F_{M} V(F_{M} \phi, w^{\ast}_{N})_{h} \\
&= F_{M} T_{N} F_{M} \phi,
\end{split}
\end{equation}
where $W^{\ast} \in X_{N}^{+}$ and $w^{\ast}_{N} \in H^{1,+}$ are the solutions, respectively, of~\cref{discrete_FP} applied to $\Phi = R_{M} \phi$ and of~\cref{collocation_hat} with $F_{M} \phi$ replacing $\phi$.
These solutions are unique under \cref{hyp-C,FsVp-to-AC}, thanks again to \cref{discr-cont-coll_eq,cont_coll_op-invertible}.

Before proving the norm convergence of $T_{N}$ to $T$, we need to extend the results of \cref{existence-uniqueness-RE} to $H^{1,+}$ in the following lemma.

\begin{lemma}\label{LNpFXsVXp-invertible}
If \cref{hyp-C,FsVp-to-AC} hold, then $(I_{X^{+}} - \mathcal{F} V^{+}) \restriction_{H^{1,+}}$ is invertible with bounded inverse.
\end{lemma}
\begin{proof}
The proof is the same as the one of \cite[Lemma 4.6]{BredaLiessi2018}, with the only difference being that $\norm{\cdot}_{X^{+}} \leq \norm{\cdot}_{H^{1,+}}$.
\end{proof}

\begin{proposition}\label{norm-convergence}
If $h\geq \tau$ and \cref{hyp-C,FsVp-to-AC,FXsVm-to-AC} hold, then $\norm{T_{N} - T}_{X \leftarrow X} \to 0$ for $N \to +\infty$.
\end{proposition}
\begin{proof}
Let $\phi \in X$ and let $w^{\ast}$ and $w^{\ast}_{N}$ be the solutions of the fixed point equations~\cref{fixed_point,collocation_hat}, respectively.
Recall that $w^{\ast}_{N}$ is a polynomial.
Recalling that $w^{\ast} = \mathcal{F} V^{+} w^{\ast} + \mathcal{F} V^{-} \phi$, it is clear that $w^{\ast} \in H^{1,+}$.
Then $(T_{N} - T) \phi
= V(\phi, w^{\ast}_{N})_{h} - V(\phi, w^{\ast})_{h}
= V^{+}(w^{\ast}_{N} - w^{\ast})_{h}$.
%\begin{equation*}
%(\hat{T}_{N} - T) \phi
%= V(\phi, w^{\ast}_{N})_{h} - V(\phi, w^{\ast})_{h}
%= V^{+}(w^{\ast}_{N} - w^{\ast})_{h}.
%\end{equation*}
By \cref{cont_coll_op-invertible}, there exists a positive integer $N_{0}$ such that, for any $N \geq N_{0}$,
\begin{equation*}
\begin{aligned}
\norm{(T_{N} - T) \phi}_{X}
&= \norm{V^{+}(w^{\ast}_{N} - w^{\ast})_{h}}_{X} \\
&\leq \norm{w^{\ast}_{N} - w^{\ast}}_{X^{+}} \\
&\leq 2 \norm{(I_{X^{+}} - \mathcal{F} V^{+})^{- 1}}_{X^{+} \leftarrow X^{+}} \norm{\mathcal{L}_{N}^{+} w^{\ast} - w^{\ast}}_{X^{+}} \\
&\leq 2 \norm{(I_{X^{+}} - \mathcal{F} V^{+})^{- 1}}_{X^{+} \leftarrow X^{+}} \norm{(\mathcal{L}_{N}^{+} - I_{X^{+}}) \restriction_{H^{1,+}}}_{X^{+} \leftarrow H^{1,+}} \norm{w^{\ast}}_{H^{1,+}}
\end{aligned}
\end{equation*}
holds by virtue of \cref{LN-I}.
Finally, from \cref{fixed-point-2} we obtain
\begin{equation*}
\norm{w^{\ast}}_{H^{1,+}} \leq \norm{((I_{X^{+}} - \mathcal{F} V^{+}) \restriction_{H^{1,+}})^{- 1}}_{H^{1,+} \leftarrow H^{1,+}} \norm{\mathcal{F} V^{-}}_{H^{1,+} \leftarrow X} \norm{\phi}_{X},
\end{equation*}
which completes the proof thanks to \cref{LNpFXsVXp-invertible}.
\end{proof}

\begin{theorem}\label{final-operator-convergence}
If $h\geq \tau$ and \cref{hyp-C,FsVp-to-AC,FXsVm-to-AC} hold, then $\norm{T_{M,N} - T}_{X \leftarrow X} \to 0$ for $N \to +\infty$ and $M\geq N-1$.
\end{theorem}
\begin{proof}
From \cref{TMN-with-FM} we have that
\begin{equation*}
T_{M,N}-T = F_{M} T_{N} (F_{M}-I_X)+(F_{M}-I_{X})T_{N} +(T_{N}-T).
\end{equation*}
Passing to the norms, the last addend converges to $0$ as $N\to +\infty$, thanks to \cref{norm-convergence}.
The second addend is zero since $M\geq N-1$, because for $h\geq \tau$ the operator $T_N$ has range in the space of polynomials of degree at most $N-1$.
In the first addend, the first $F_M$ is applied again to polynomials of degree at most $N-1\leq M$ and thus can be neglected.
Observe that $T_N$ is bounded, thanks to \cref{T-bounded,norm-convergence}.
From the convergence of the Fourier projection in $L^2$, if $\phi\in X$, then $\norm{(F_{M}- I_X) \phi}_{X} \to 0$ for $M \to +\infty$.
By the Banach--Steinhaus theorem, the sequence $\norm{F_{M} - I_X}_{X \leftarrow X}$ is bounded, hence
\begin{equation*}
\norm{F_{M} - I_X}_{X\leftarrow X} \xrightarrow[M \to +\infty]{} 0,
\end{equation*}
which completes the proof.
\end{proof}

\section{Lyapunov exponents}
\label{sec:le}

\subsection{Definition}
\label{sec:defLE}

In order to define and characterize the LEs for REs of type \cref{RE} we resort to the same arguments used in \cite{Breda2010,BredaVanVleck2014} for DDEs, which we summarize next.

The starting point is the key result \cite[Corollary 2.2]{Ruelle1982}, which extends to compact operators in infinite-dimensional Hilbert spaces the celebrated Oseledets's multiplicative ergodic theorem \cite{Oseledets1968} for finite-dimensional dynamical systems (see also \cite{ChiconeLatushkin1999,EdenFoiasTemam1991}).
Based on this result, (eventual) compactness of the evolution family $\{T(t,s)\}_{t\geq s}$ for $T(t,s)$ introduced in \cref{T} ensures the existence of a possibly infinite sequence of LEs of \cref{RE}.
In particular, they are defined as the logarithm of the eigenvalues of the operator $\Lambda\colon X\to X$ given by%
\footnote{Let us remark that this definition relies upon an ergodicity assumption considered in \cite[Corollary 2.2]{Ruelle1982}.
This condition guarantees regularity of the underlying dynamical system, which in turns gives the LEs as exact limits.
Other definitions based on $\liminf$ and $\limsup$ can be invoked, as is implicitly the case of the following characterization.
For a more detailed discussion on these aspects we refer to \cite[Section 6]{Breda2010} and to the references therein.}
\begin{equation*}
\Lambda\coloneqq\lim_{n\to+\infty}[T(s+n\tau,s)^{\ast}T(s+n\tau,s)]^{1/2n}.
\end{equation*}
Recall that in \cref{sec:compactness} we proved that $T(s+h,s)$ is compact for any $h\geq\tau$.
Here we characterize the LEs through an orthogonal factorization of the evolution family.
This characterization follows the standard DQR technique developed for linear and nonautonomous ODEs \cite{DieciVanVleck2005,DieciVanVleck2006,DieciVanVleck2008}, extended to the infinite-dimensional state space $X$ (which is indeed a separable Hilbert space).

Let $\{t_{n}\}_{n\in\mathbb{N}}$ be given through $t_{n}\coloneqq s+n\tau$ and define for brevity
\begin{equation}\label{Tninf}
T_{n}\coloneqq T(t_{n+1},t_{n}).
\end{equation}
Starting from a given (randomly chosen \cite{BenettinGalganiGiorgilliStrelcyn1980a,DieciElia2006}) unitary operator $Q_{0}\colon X\to X$, iteratively compute the sequence of QR factorizations%
\footnote{\label{convention}%
Throughout the paper we assume that the upper triangular factor of any QR factorization has non-negative diagonal elements, guaranteeing uniqueness of the QR factorization.
We also assume the convention $\log(0)=-\infty$.}
\begin{equation}\label{DQRinf}
Q_{n+1}R_{n}=T_{n}Q_{n}.
\end{equation}
Let $[R_{n}]_{ii}$, $i\in\mathbb{N}$, denote the $i$-th diagonal entry of the $n$-th triangular factor, according to the basis of Legendre polynomials in $X$.
The Lyapunov spectrum
\begin{equation}\label{spLEinf}
\Sigma_{\text{LE}}=\bigcup_{i=0}^{+\infty}[\alpha_{i},\beta_{i}]
\end{equation}
is characterized by the lower and upper LEs
\begin{equation}\label{LEinf}
\alpha_{i}\coloneqq\liminf_{n\to+\infty}\frac{1}{t_{n}}\sum_{k=0}^{n-1}\log([R_{k}]_{ii}),\quad\beta_{i}\coloneqq\limsup_{n\to+\infty}\frac{1}{t_{n}}\sum_{k=0}^{n-1}\log([R_{k}]_{ii}).
\end{equation}
This characterization provides the basis for the numerical approximation of (a number of) Lyapunov spectral intervals.
Indeed, the proposed approach relies on truncating the limits in \cref{LEinf} to a sufficiently large $n$ and on substituting each $T_{n}$ in the DQR sequence \cref{DQRinf} with a suitable finite-dimensional approximation.
This approximation is obtained via collocation and Fourier projection as illustrated in \cref{sec:discr}.
Note that the same technique can be used to approximate the Sacker--Sell spectrum (as done for DDEs in \cite{BredaVanVleck2014}, to which we refer for the relevant details).

\subsection{Approximation}
\label{sec:apprLE}

Consider the sequence of evolution operators $T_{n}\colon X\to X$ in \cref{Tninf} and the iterative QR factorization \cref{DQRinf}.
Given $N\in\mathbb{N}$ and $M\geq N-1$, for ease of notation let
\begin{equation}\label{Tnbar}
\overline{T}_{n}\coloneqq(T_{n})_{M,N}\colon X\to X
\end{equation}
be the infinite-dimensional finite-rank approximation of $T_{n}$ obtained as described in \cref{sec:discr}.
Consider the relevant iterative QR factorization
\begin{equation}\label{DQRbar}
\overline{Q}_{n+1}\overline{R}_{n}=\overline{T}_{n}\overline{Q}_{n}.
\end{equation}
Finally, let $\overline{\mathbf{T}}_{n}\coloneqq X_{M}\to X_{M}$ be the matrix corresponding to $\overline{T}_{n}$ as introduced in \cref{TMNbold} and consider the finite-dimensional iterative QR factorization
\begin{equation}\label{DQRbold}
\overline{\mathbf{Q}}_{n+1}\overline{\mathbf{R}}_{n}=\overline{\mathbf{T}}_{n}\overline{\mathbf{Q}}_{n}.
\end{equation}
The three different QR sequences \cref{DQRinf}, \cref{DQRbar} and \cref{DQRbold} lead, respectively, to the three Lyapunov spectra \cref{spLEinf},
\begin{equation}\label{spLEbar}
\overline{\Sigma}_{{\rm LE}}=\bigcup_{i=0}^{+\infty}[\overline{\alpha}_{i},\overline{\beta}_{i}],
\end{equation}
and
\begin{equation}\label{spLEbold}
\overline{\bm{\Sigma}}_{{\rm LE}}=\bigcup_{i=0}^{d(M+1)-1}[\overline{\bm{\alpha}}_{i},\overline{\bm{\beta}}_{i}]
\end{equation}
characterized, respectively, by \cref{LEinf},
\begin{equation}\label{LEbar}
\overline{\alpha}_{i}=\liminf_{n\to+\infty}\frac{1}{t_{n}}\sum_{k=0}^{n-1}\log([\overline{R}_{k}]_{ii}),\quad\overline{\beta}_{i}=\limsup_{n\to+\infty}\frac{1}{t_{n}}\sum_{k=0}^{n-1}\log([\overline{R}_{k}]_{ii}),
\end{equation}
and
\begin{equation}\label{LEbold}
\overline{\bm{\alpha}}_{i}=\liminf_{n\to+\infty}\frac{1}{t_{n}}\sum_{k=0}^{n-1}\log([\overline{\mathbf{R}}_{k}]_{ii}),\quad\overline{\bm{\beta}}_{i}=\limsup_{n\to+\infty}\frac{1}{t_{n}}\sum_{k=0}^{n-1}\log([\overline{\mathbf{R}}_{k}]_{ii}).
\end{equation}
We refer to \cref{LEinf} as the \emph{exact} (upper and lower) LEs (i.e., for $T_{n}$ on $X$), to \cref{LEbar} as the \emph{approximated} LEs (i.e., for $\overline{T}_{n}$ on $X$) and to \cref{LEbold} as the \emph{computable} LEs (i.e., for $\overline{\mathbf{T}}_{n}$ on $X_{M}$).
Indeed, truncating the limits in \cref{LEbold} gives rise to actually computable quantities, whose convergence to their exact counterpart is discussed in the following \lcnamecref{sec:conv-exp}.

\subsection{Convergence}
\label{sec:conv-exp}

To relate the computable exponents \cref{LEbold} to (some of) the exact ones \cref{LEinf} we follow the same steps of the analysis in \cite[Section 6]{BredaVanVleck2014}.
Here we restrict to summarizing only the main ideas, avoiding a detailed repetition, especially for the most technical parts.
Indeed, the whole analysis is unchanged since it is based on the underlying evolution family $\{T(t,s)\}_{t\geq s}$ and, in particular, on the discrete-time sequence $\{T_{n}\}_{n\in\mathbb{N}}$ introduced in \cref{Tninf}.

We start by showing that the computable exponents \cref{LEbold} coincide with the non-degenerate%
\footnote{Degenerate exponents are those at $-\infty$, see also the end of \cref{convention}.}
ones in \cref{LEbar}.
Then the non-degenerate exponents in \cref{LEbar} approximate finitely-many exact exponents \cref{LEinf}, with relevant errors depending on $\norm{\overline{T}_{n}-T_{n}}_{X \leftarrow X}$, strongly motivating the analysis developed in \cref{sec:conv-op}, see in particular \cref{final-operator-convergence}.

In what follows, by using the basis of Legendre polynomials in $X$, we describe a linear operator $B\colon X\to X$ block-wise as
\begin{equation*}
B=\begin{pmatrix}
(B)_{11}&(B)_{12}\\
(B)_{21}&(B)_{22}
\end{pmatrix},
\end{equation*}
where $(B)_{11}\in(\mathbb{R}^{d})^{(M+1)\times(M+1)}$, $(B)_{12}\in(\mathbb{R}^{d})^{(M+1)\times\infty}$, $(B)_{21}\in(\mathbb{R}^{d})^{\infty\times(M+1)}$ and $(B)_{22}\in(\mathbb{R}^{d})^{\infty\times\infty}$.
Accordingly, blocks of zeros and identity blocks will be denoted respectively as $0_{ij}$ and $I_{ij}$, $i,j\in\{1,2\}$.
Above, $M\in\mathbb{N}$ implicitly refers to the one used to construct $\overline{T}_{n}$ in \cref{Tnbar}.
Recall that $\overline{T}_{n}$ has finite rank since $h=\tau$.
Consequently,
\begin{equation*}%\label{Tnbarform}
\overline{T}_{n}=
\begin{pmatrix}
\overline{\mathbf{T}}_{n}&0_{12}\\
0_{21}&0_{22}
\end{pmatrix},
\end{equation*}
leading to the following result.
\begin{lemma}[\protect{\cite[Lemma 1]{BredaVanVleck2014}}]%\label{lemQR}
Let $\overline{Q}_{0}\colon X\to X$ be unitary.
If $\overline{\mathbf{T}}_{0}(\overline{Q}_{0})_{11}$ is invertible, then
\begin{equation*}%\label{Barnstorm}
\overline{Q}_{n+1}=
\begin{pmatrix}
(\overline{Q}_{n+1})_{11}&0_{12}\\
0_{21}&I_{22}
\end{pmatrix},\quad
\overline{R}_{n+1}=
\begin{pmatrix}
(\overline{R}_{n+1})_{11}&0_{12}\\
0_{21}&0_{22}
\end{pmatrix}
\end{equation*}
for any $n\in\mathbb{N}$.
\end{lemma}
\noindent As a consequence, we obtain $(\overline{R}_{n+1})_{11}=\overline{\mathbf{R}}_{n+1}$ and $(\overline{Q}_{n+1})_{11}=\overline{\mathbf{Q}}_{n+1}$, leading to the next theorem.
\begin{theorem}[\protect{\cite[Theorem 2]{BredaVanVleck2014}}]%\label{compSS}
Let $m\coloneqq d(M+1)$.
The first $m$ Lyapunov intervals in \cref{spLEbar} coincide with those in \cref{spLEbold}, while the remaining ones are clustered at $-\infty$:
\begin{equation*}
[\overline{\alpha}_{i},\overline{\beta}_{i}]=
\begin{cases}
[\overline{\bm{\alpha}}_{i},\overline{\bm{\beta}}_{i}]&\text{for } i\in\{0,\dots,m-1\},\\
\{-\infty\}&\text{for } i\geq m.
\end{cases}
\end{equation*}
\end{theorem}
\noindent The next final result relates the exact QR sequence \cref{DQRinf} to the finite-rank one \cref{DQRbar} via perturbations of the finite-rank triangular factors in a backward error analysis framework.
These perturbations are explicitly related to the difference $\overline{T}_{n}-T_{n}$, for which the norm convergence is obtained in \cref{sec:conv-op}, see \cref{final-operator-convergence}.
\begin{lemma}[\protect{\cite[Lemma 2]{BredaVanVleck2014}}]%\label{LAarg}
Let $\overline{Q}_{0}=Q_{0}\colon X\to X$ be unitary.
If $\overline{\mathbf{T}}_{0}(\overline{Q}_{0})_{11}$ is invertible, then
\begin{equation*}
Q_{n+1}R_{n}\cdots R_{0}=\overline{Q}_{n+1}[\overline{R}_{n}+E_{n}]\cdots[\overline{R}_{0}+E_{0}]
\end{equation*}
with
\begin{equation}\label{Ek}
E_{k}\coloneqq-\overline{Q}_{k+1}^{\ast}(\overline{T}_{k}-T_{k})\overline{Q}_{k},\quad k\in\{0,\dots,n\}.
\end{equation}
\end{lemma}

The analysis is now completed by relating the perturbed finite-rank triangular factors to another QR sequence, viz.
\begin{equation*}%\label{DQRtilde}
\widetilde{Q}_{n+1}\widetilde{R}_{n}=[\overline{R}_{n}+E_{n}]\widetilde{Q}_{n}
\end{equation*}
starting from $\widetilde{Q}_{0}=I_{X}$ and with $\widetilde{Q}_{n}\colon X\to X$ near identity unitary operators for positive $n$.
The existence of the sequence $\{\widetilde{Q}_{n}\}_{n\in\mathbb{N}}$ is first translated into a zero finding problem on a suitable Banach space.
Then the application of a Newton--Kantorovich type argument provides the final bounds on the difference between approximated and exact exponents, i.e., $\abs{\overline{\alpha}_{i}-\alpha_{i}}$ and $\abs{\overline{\beta}_{i}-\beta_{i}}$.
These bounds are shown to depend on the distance of $\widetilde{Q}_{n}$ from $I_{X}$, as well as on norm bounds on the backward errors $E_{k}$ in \cref{Ek} (and hence on $\norm{\overline{T}_{k}-T_{k}}_{X \leftarrow X}$) and, finally, on the magnitude of the diagonal entries $[R_{k}]_{ii}$ of the upper triangular factors (see \cite[Section 6.2 and Proposition 2]{BredaVanVleck2014}).
As this part is rather technical, yet identical to the one in \cite{BredaVanVleck2014}, we omit it and refer to the cited work.
Therein, the interested reader can also find useful comments on the (im)possibility of guaranteeing the approximation of the dominant exponents, in relation to the notion of integral separation and to the role of the chosen coordinate system.
Let us just observe that the above mentioned bounds can be recovered from the computed triangular factors $\overline{\mathbf{R}}_{k}$ and from the convergence of $\norm{\overline{T}_{k}-T_{k}}_{X \leftarrow X}$ from \cref{final-operator-convergence}.
For a general perturbation theory on the approximation of LEs and stability spectra via QR methods for linear operators on Hilbert spaces see \cite{BadawyVanVleck2012}.

\section{Implementation and numerical results}
\label{sec:num}

For the discretization of the evolution operators, we use the MATLAB code \texttt{eigTMNc}, available at \url{https://cdlab.uniud.it/software}, which implements the discretization method of \cite{BredaLiessi2018,BredaLiessi2020}.
That method is based on the pseudospectral collocation of the elements of $X$ on the Chebyshev type II nodes (extrema) and of the elements of $X^+$ on the Chebyshev type I nodes (zeros).
As we observe next, the differences with the method illustrated in \cref{sec:discr} are unimportant as long as we are concerned with the approximation of LEs.
Indeed, representing the state of the system in $X$ via Chebyshev collocation or truncation of the Fourier series in principle leads to different functions.
However, the DQR method, as observed below and after \cref{Tninf}, starts from a random matrix, rendering the precise initial value irrelevant.
Moreover, since $h=\tau$, the only relevant elements of $X^+$ are polynomials of degree at most $N-1$, so collocating at the Chebyshev and the Legendre zeros gives the same polynomial in both cases.

The matrix representation of the operator $\mathbf{T}_{M,N}$ is based on the following reformulation.
Recalling \cref{TMNbold}, by virtue of~\cref{decomp_V}, we can write
\begin{equation*}
\mathbf{T}_{M, N} \Phi = \mathbf{T}_{M}^{(1)} \Phi + \mathbf{T}_{M, N}^{(2)} W^{\ast},
\end{equation*}
with $\mathbf{T}_{M}^{(1)} \colon X_{M} \to X_{M}$ and $\mathbf{T}_{M, N}^{(2)} \colon X_{N}^{+} \to X_{M}$ defined as
\begin{equation*}
\mathbf{T}_{M}^{(1)} \Phi \coloneqq R_{M} (V^{-} P_{M} \Phi)_{h}, \quad\quad
\mathbf{T}_{M, N}^{(2)} W \coloneqq R_{M} (V^{+} P_{N}^{+} W)_{h}.
\end{equation*}
Similarly, the fixed point equation~\cref{discrete_FP} can be rewritten as
\begin{equation*}
(I_{X_{N}^{+}} - \mathbf{U}_{N}^{(2)}) W = \mathbf{U}_{M, N}^{(1)} \Phi,
\end{equation*}
with $\mathbf{U}_{M, N}^{(1)} \colon X_{M} \to X_{N}^{+}$ and $\mathbf{U}_{N}^{(2)} \colon X_{N}^{+} \to X_{N}^{+}$ defined as
\begin{equation*}
\mathbf{U}_{M, N}^{(1)} \Phi \coloneqq R_{N}^{+} \mathcal{F} V^{-} P_{M} \Phi, \quad\quad
\mathbf{U}_{N}^{(2)} W \coloneqq R_{N}^{+} \mathcal{F} V^{+} P_{N}^{+} W.
\end{equation*}
Since $I_{X_{N}^{+}} - U_{N}^{(2)}$ is invertible, the operator $\mathbf{T}_{M, N} \colon X_{M} \to X_{M}$ can be eventually reformulated as
\begin{equation*}%\label{matrixTMN}
\mathbf{T}_{M, N} = \mathbf{T}_{M}^{(1)} + \mathbf{T}_{M, N}^{(2)} (I_{X_{N}^{+}} - \mathbf{U}_{N}^{(2)})^{-1} \mathbf{U}_{M, N}^{(1)}.
\end{equation*}
Observe that the matrices $\mathbf{T}_{M}^{(1)}$ and $\mathbf{T}_{M, N}^{(2)}$ do not depend on the right-hand side of \cref{RE}.

For the approximation of the LEs we apply the DQR method to the evolution operators defined on a uniform grid of step size $\tau$ on the time interval $[\tau,t_{\text{f}}]$, where $t_{\text{f}}\geq\tau$ is the final time for the truncation of the limits in \cref{LEbold}.
The approximating sequence starts from a random unitary matrix.
At each step the evolution operator is applied to the unitary matrix and the QR factorization of the resulting matrix is computed: the R factor contributes to the approximation of the LEs via \cref{LEbold}, while the Q factor becomes the unitary matrix for the next step.

The MATLAB codes implementing the (full) method and the scripts reproducing the experiments presented below are available at \url{https://cdlab.uniud.it/software}.

\bigskip

Let us recall that the focus of this work is on the convergence proof and on the relevant theoretical background.
The outcome of the experiments performed with the current method are the same as those obtained with the method of \cite{BredaLiessi2024} for all the test cases considered therein.
Therefore, to avoid repetition, we limit ourselves to report only on the results regarding the RE with quadratic nonlinearity from \cite{BredaDiekmannLiessiScarabel2016}:
\begin{equation}\label{quadRE}
x(t) = \frac{\gamma}{2}\int_{-3}^{-1} x(t+\theta)(1-x(t+\theta)) \D\theta.
\end{equation}
For details on the properties of the equation we refer to \cite[Section 5.1]{BredaLiessi2024}.
As done there, as reference values for the LEs, we use the spectra of the evolution operators as described therein, approximating them here with \texttt{eigTMNc}.

\Cref{quadRE} is nonlinear and in order to compute the LEs we need to linearize it around a (generic) solution.
We compute this solution using the trapezoidal method described in \cite{MessinaRussoVecchio2008}, using $r=40$ pieces per unit time in the numerical solution, starting from a constant initial value of $0.1$.
As for the linearization, observe that the right-hand side of \cref{quadRE} is not Fr\'echet-differentiable as a map from $L^2$ to $\mathbb{R}$.
However, we can consider the linear RE
\begin{equation*}%\label{quadraticrelin}
x(t) = \frac{\gamma}{2} \int_{-3}^{-1} (1-2\bar{x}(t+\theta)) x(t+\theta) \D\theta.
\end{equation*}
We refer to \cite{BredaLiessiVermiglio2022} and to \cite[Section 3.5]{DiekmannGettoGyllenberg2008} for more details about linearizations in the case of REs formulated on an $L^1$ state space.
Note that the same problems arise in the case of any $L^p$ state space.

The first experiment investigates the dependence of the errors on the LEs on the choice of $M=N$ and of the final time $t_{\text{f}}$.
As in \cite{BredaLiessi2024}, we choose values of $\gamma$ corresponding to the stable trivial equilibrium ($\gamma=0.5$), the stable nontrivial equilibrium ($\gamma=3$) and the stable periodic orbit ($\gamma=4$).
The reference values are obtained by using \texttt{eigTMNc} with the default options and $120$ as the degree of the collocation polynomials (fixed independently of $M=N$).
We obtain qualitatively the same results as in \cite{BredaLiessi2024}, namely that the LEs converge linearly in $t_{\text{f}}$ for fixed $M=N$ (\cref{fig:quad-le-T-1000}), and that no convergence is observed in $M=N$ for fixed $t_{\text{f}}$ (\cref{fig:quad-le-M}).
We restate here that our interpretation of the latter phenomenon is that we are observing an error barrier due to the truncation of the time limits in \cref{LEbold}, and apparently that truncation error dominates the error due to the discretization of the evolution operators.

\begin{figure}
\begin{center}
\includegraphics{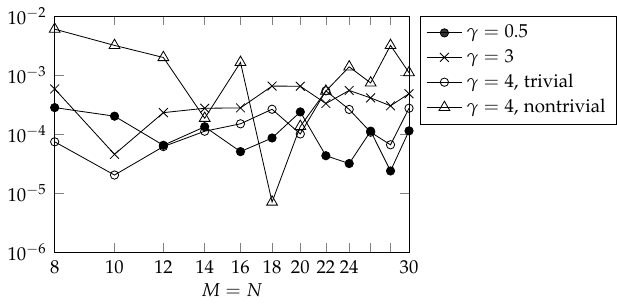}
\caption{Absolute errors on the dominant LEs of the RE with quadratic nonlinearity \cref{quadRE} for values of $\gamma$ corresponding to the stable trivial equilibrium ($\gamma=0.5$), the stable nontrivial equilibrium ($\gamma=3$) and the stable periodic orbit ($\gamma=4$).
For the last one, both the trivial and the dominant nontrivial exponents are shown.
The errors are measured with respect to the exponents computed via \texttt{eigTMNc}.
The final time is $t_{\text{f}}=1000$.}
\label{fig:quad-le-M}
\end{center}
\end{figure}

\begin{figure}
\begin{center}
\includegraphics{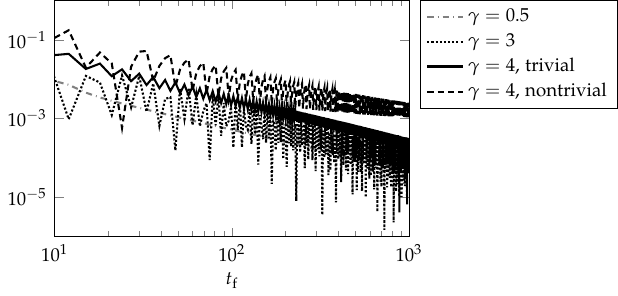}
\caption{Absolute errors on the dominant LEs of the RE with quadratic nonlinearity \cref{quadRE} for values of $\gamma$ corresponding to the stable trivial equilibrium ($\gamma=0.5$), the stable nontrivial equilibrium ($\gamma=3$) and the stable periodic orbit ($\gamma=4$).
For the last one, both the trivial and the dominant nontrivial exponents are shown.
The errors are measured with respect to the exponents computed via \texttt{eigTMNc}.
The exponents are computed for $M=N=16$.}
\label{fig:quad-le-T-1000}
\end{center}
\end{figure}

The second experiment consists in the computation of the diagram of the first two dominant LEs of \cref{quadRE} when varying $\gamma$, shown in \cref{fig:quad-le-diagram} (cf.\ \cite[Figure 2.3]{BredaDiekmannLiessiScarabel2016} and \cite[Figure 4]{BredaLiessi2024}).
The LEs are computed with $M=N=15$ and $t_{\text{f}}=1000$.
The diagram is very close to the ones obtained in the previous cited works.
In particular, we can observe the Hopf bifurcation at $\gamma=2+\frac{\pi}{2}$, the period-doubling bifurcations at $\gamma\approx 4.33, 4.50, 4.53$ (cf.\ $\gamma\approx 4.32, 4.49, 4.53$ in \cite{BredaLiessi2024}), and the insurgence of chaos at $\gamma\geq4.55$ (as in \cite{BredaLiessi2024}).
We also observe a stability island (of stable periodic solutions) starting at $\gamma\approx4.8665$ (as in \cite{BredaLiessi2024}) and the period-doubling bifurcations at $\gamma\approx4.8800,4.8865$ (cf.\ $\gamma\approx4.8795,4.8860$ in \cite{BredaLiessi2024}) leading back to chaos starting at $\gamma\approx4.8880$ (cf.\ $\gamma\approx4.8885$ in \cite{BredaLiessi2024}).

\begin{figure}
\begin{center}
\includegraphics{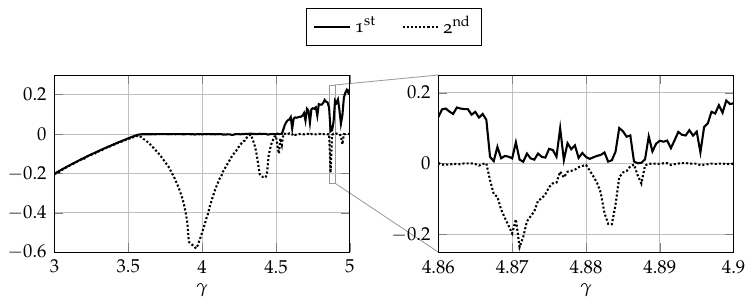}
\caption{Diagram of the first two dominant (in descending order) LEs of the RE with quadratic nonlinearity \cref{quadRE} when varying $\gamma$, computed with $M=N=15$ and $t_{\text{f}}=1000$.}
\label{fig:quad-le-diagram}
\end{center}
\end{figure}

\section{Concluding remarks}
\label{sec:concl}

The method we used to approximate the evolution operators of REs has been adapted also to systems of coupled DDEs and REs in \cite{BredaLiessi2020} and the \texttt{eigTMNc} implementation can be readily applied to them as well.
In fact, in our previous work on LEs \cite{BredaLiessi2024} we provided an example of a coupled system consisting in a simplified version of the Daphnia model with a logistic growth of the resource \cite{BredaDiekmannMasetVermiglio2013}.

In principle, the structure of the method we presented here should work equally well for coupled systems.
However, from the point of view of the theoretical foundations of the method, one should prove the compactness of evolution operators of coupled systems, which may present some difficulties arising from the coupling (cf. \cite{BredaLiessi2020} for the case of Floquet theory).
We thus reserve this topic for future work, with the strong motivation that examples of coupled systems are widely used in modeling population dynamics: besides the simplified Daphnia model cited above and the very complex original Daphnia model \cite{DiekmannGyllenbergMetzNakaokaDeRoos2010}, we mention \cite{RipollFont2023,LythgoePellisFraser2013}, but other examples can be found in \cite{DiekmannGettoGyllenberg2008,DiekmannGyllenberg2012} and the references therein.

\section*{Acknowledgments}

The authors are members of INdAM research group GNCS and of UMI research group ``Mo\-del\-li\-sti\-ca socio-epidemiologica''.
This work was partially supported by the Italian Ministry of University and Research (MUR) through the PRIN 2020 project (No.\ 2020JLWP23) ``Integrated Mathematical Approaches to Socio-Epidemiological Dynamics'', Unit of Udine (CUP G25F22000430006), and by the GNCS 2023 project ``Sistemi dinamici e modelli di evoluzione: tecniche funzionali, analisi qualitativa e metodi numerici'' (CUP: E53C22001930001).

{\sloppy
\printbibliography
\par}

\end{document}